\documentclass[11pt]{amsart}

\usepackage{amsmath, amssymb}
\usepackage{enumerate}

\begin{document}

\reversemarginpar






\newtheorem{theorem}{Theorem}[section]
\newtheorem{corollary}[theorem]{Corollary}
\newtheorem{lemma}[theorem]{Lemma}
\newtheorem{prop}[theorem]{Proposition}
\newtheorem*{thmA}{Theorem A}
\newtheorem*{thmB}{Theorem B}
\theoremstyle{definition}
\newtheorem{example}[theorem]{Example}
\newtheorem{question}[theorem]{Question}
\newtheorem{examples}[theorem]{Examples}
\newtheorem{conj}[theorem]{Conjecture}
\newtheorem{rem}[theorem]{Remark}
\newtheorem{rems}[theorem]{Remarks}
\newtheorem{definition}[theorem]{Definition}
\newtheorem{definitions}[theorem]{Definitions}

\newcommand{\itemizespacing}{\addtolength{\itemsep}{0.5\baselineskip}}

\newcommand{\codim}{\operatorname{codim}}
\newcommand{\dist}{\operatorname{dist}}
\newcommand{\spann}{\operatorname{span}}
\newcommand{\C}{\mathbb{C}}

\newcommand{\Y}{Y}

\newcommand{\K}{\mathcal{K}}

\newcommand{\D}{\mathbb{D}}
\newcommand{\N}{\mathbb{N}}
\newcommand{\R}{\mathbb{R}}
\newcommand{\T}{\mathbb{T}}
\newcommand{\Z}{\mathbb{Z}}
\newcommand{\A}{\mathbb{A}}
\newcommand{\Der}{\operatorname{Der}}
\newcommand{\diag}{\operatorname{diag}}
\newcommand{\esssup}{\operatorname{ess\; sup}}
\newcommand{\inner}[1]{\langle #1 \rangle}

\newcommand{\pt}{\partial}

\newcommand{\dbar}{\bar{\partial}}

\newcommand{\cB}{\mathcal{B}}

\newcommand{\cD}{\mathcal{D}}
\newcommand{\cH}{\mathcal{H}}
\newcommand{\cJ}{\mathcal{J}}
\newcommand{\cG}{\mathcal{G}}
\newcommand{\cQ}{\mathcal{Q}}
\newcommand{\cO}{\mathcal {O}}

\newcommand{\tpsi}{\tilde\psi}

\newcommand{\Hdot}{H_0}
\newcommand{\h}{H}
\newcommand{\hdot}{H_0}

\newcommand{\aaaa}{a}
\newcommand{\wh}{\widehat}
\newcommand{\uu}{u}
\newcommand{\vv}{v}
\newcommand{\GH}{g}
\newcommand{\toe}{\tau}

\newcommand{\beqn}{\begin{equation}}
\newcommand{\neqn}{\end{equation}}

\newcommand{\Carl}{\operatorname{Carleson}}
\newcommand{\G}{\operatorname{G}}
\renewcommand{\Re}{\operatorname{Re}}
\renewcommand{\Im}{\operatorname{Im}}
\renewcommand{\vec}[1]{{\bf #1}}
\newcommand{\Mod}{\operatorname{Mod}}
\newcommand{\Modloc}{\Mod^\text{loc}}
\newcommand{\Etwo}{E^2}

\newcommand{\al}{\alpha}
\newcommand{\be}{\beta}
\newcommand{\de}{\delta}
\newcommand{\ga}{\gamma}
\newcommand{\eps}{\varepsilon}
\newcommand{\la}{\lambda}
\newcommand{\si}{\sigma}

\newcommand{\sm}{\setminus}
\newcommand{\wt}{\widetilde}

\newcommand{\Ga}{\Gamma}
\newcommand{\Om}{\Omega}

\newcommand{\supp}{\operatorname{supp}}
\newcommand{\Hs}{\mathcal{H}^s}
\newcommand{\Hsi}{\mathcal{H}^\sigma}
\newcommand{\Hone}{\mathcal{H}^1}
\newcommand{\Hscont}{{\mathcal{H}}^s_{\rm cont}}
\newcommand{\Honecont}{{\mathcal{H}}^1_{\rm cont}}

\newcommand{\Omc}{\Om^c}
\newcommand{\LtwomuH}{L^2(\mu, \Hdot)}


\title{Spectral analysis of 2D outlier layout}

\author{Mihai Putinar}
\address[mputinar@math.ucsb.edu, mihai.putinar@ncl.ac.uk]{University of California at Santa Barbara, CA,
USA and Newcastle University, Newcastle upon Tyne, UK}

\begin{abstract} Thompson's partition of a cyclic subnormal operator into normal and completely non-normal components is combined with a non-commutative calculus for hyponormal operators for separating outliers from the cloud, in rather general point distributions in the plane. The main result provides exact transformation formulas from the power moments of the prescribed point distribution into the moments of the uniform mass carried by the cloud. The proposed algorithm solely depends on the Hessenberg matrix associated to the original data. The robustness of the algorithm is reflected by the insensitivity of the output under trace class, or by a theorem of Voiculescu, under certain Hilbert-Schmidt class, additive perturbations of the Hessenberg matrix.

\end{abstract}

\subjclass[2010]{
Primary: 47B20; Secondary 33C45, 44A60, 47N30.}

\keywords{
Subnormal operator, orthogonal polynomial, trace formula, outlier, cloud, polynomial approximation, Christoffel-Darboux kernel.}

\maketitle

\section{Introduction}

The Lebesgue decomposition of a positive measure supported on the real line or the circle and 
the spectral analysis of symmetric or unitary Hilbert space linear operators cannot be dissociated. We take for granted nowadays 
operator arguments and Hilbert space geometry features of scattering theory, dynamical systems or approximation schemes, all generally
reflecting the fine structure of some underlying spectral measure.

The association of Lebesgue decomposition of measures in higher dimensions with spectral analysis is much less studied and developed.
It naturally appeared in perturbation theory questions
\cite{Voiculescu-1979,Voiculescu-1980,Makarov-1986} and it also proved to be essential in some elaborate classification results concerning 
classes of non-normal operators \cite{Pincus-1968,Helton-Howe-1975,Thomson-1991}. In the present note we exploit such
a classical, but not widely circulated chapter of modern operator theory with the specific aim at bringing forward an abstract tool with some statistical flavor, namely
detecting ``outliers'' (that is the discrete part) from the ``cloud'' (the continuous, or 
better, 2D-absolutely continuous part) of a measure in the real plane. From the mathematical point of view the ``subnormal dissection'' of planar measures we 
propose is more convoluted, and hence more interesting. We are well aware of the ambiguity of terminology when mentioning outliers \cite{Hawkins-1980}.

The case of two dimensions is special, not last because two real coordinates can be arranged into a single complex variable. Also the spectrum, broadly understood as the tangible numerical support of a linear transform, 
is contained in the complex plane, as well as its localized resolvents, usually interpreted as generalized Cauchy transforms of innate data. This immediately leads to the "black box" approach of restoring the whole from indirect measurements; we tacitly adopt such a perspective, familiar in more applied sciences.

The mathematical construct we propose can be summarized as follows. A compactly supported, positive measure $\mu$ in the complex plane is given. The closure of
complex polynomials in the associated Lebesgue $L^2$-space is a Hilbert space whose inner product solely dependens on the {\it real} power moments of it.
 The multiplier $S$ by the complex variable on this space is a cyclic subnormal operator whose spectrum contains the support of the generating measure. A landmark result of Thomson \cite{Thomson-1991} decomposes this subnormal operator into orthogonal matrix blocks. One of these blocks, the normal component, gathers the point masses and some singular parts of $\mu$. The other blocks are irreducible and fill with their spectrum some of the connected components of the complement of the support of
$\mu$. They provide a dissection of $\mu$ into mutual disjoint essential parts. In order to constructively identify these continuous components of $\mu$ we convert the moment data following a non-commutative calculus. Central to this step is Helton and Howe formula which relates traces of commutators of non-commutative polynomials in $S$ and $S^\ast$ to the principal function of $S$, a spectral invariant proposed by Pincus half a century ago \cite{Pincus-1968,Carey-Pincus-1981}. Trace class estimates in this situation go back to Berger and Shaw \cite{Berger-Shaw-1974} and Voiculescu \cite{Voiculescu-1980'}. Finally, the conditioned moments are organized into an exponential transform which plays the role of ``equilibrium potential'' of the possibly thickened, continuous part of $\mu$. The exponential transform is a superharmonic function decreasing on the complement from the value $1$ at infinity, to the value $0$, and behaving close to the boundary of these components as euclidean distance. See \cite{Gustafsson-Putinar-2017} for a recent survey of basic properties of the exponential transform. All in all, putting aside technical ingredients, we offer to the practitioner an algorithm for separating a simply connected cloud, or a union of of simply connected clouds, from scattered outliers, which can be points, non-closed curves or even more complicated "thin" shapes; all in terms of moment 
data. The complex orthogonal polynomials and the associated Hessenberg matrix representing $S$ naturally enter into computations. A different non-commutative perspective on outlier identification recently appeared in  \cite{Belinschi-2017} for a rather different path.

Omnipresent in our study, although not always explicit, are bounded point evaluations and reproducing kernels associated to some functional Hilbert spaces. For $L^2$ spaces of polynomials of a prescribed degrees, these are the well known Christoffel-Darboux kernels. In the classical one variable setting, the fine analysis and in particular asymptotics of these kernels provide the main separation tool between the 1D-absolute continuous part and singular part of a measure. In this direction an informative account is Nevai's eulogy of Geza Freud work \cite{Nevai-1986}. The ultimate results, with far reaching consequences to approximation theory and beyond, are due to Totik and collaborators \cite{Totik-2000,Mate-Nevai-Totik-1991}.  For a novel application to ergodic theory see \cite{Korda-Mezic-Putinar-2018}.
An operator point of view to the same topics, with unexpected applications to spectral theory, was recently advocated by Simon \cite{Simon-2008,Simon-2009,Simanek-2012}. 
Not surprising, the multivariate analog of the Christoffel-Darboux kernel analysis is only nowadays slowly developing \cite{Bos-DelleVecchia-Mastroiani-1998,Bloom-Levenberg-2008,Kroo-Lubinsky-2013a,Kroo-Lubinsky-2013b,BPSS-2018}. On this ground, potential benefits to the statistics of point distributions are emerging \cite{Lasserre-Pauwels-2019,BPSS-2018,Lasserre-Pauwels-Putinar-2018}. 

The present note remains at the purely theoretical level. Computational and numerical aspects will be addressed in a forthcoming article.
Having in mind such a sequel and its potential wider audience prompted us to balance the text and briefly comment a rather comprehensive bibliography.

\section{Thomson's partition of a measure}

Let $H$ be a separable, complex Hilbert space. A linear bounded operator $S \in {\mathcal L}(H)$ is called {\it subnormal} if there exists a larger
Hilbert space $H \subset K$ and a normal operator $N \in {\mathcal L}(K)$, such that $N$ leaves invariant $H$ and $N|_H = S$. Subnormal operators 
are well studied, mainly with function theory tools, as the quintessential example is offered by multiplication by the complex variable on a Hilbert space of analytic functions.
A subnormal operator $S$ is called irreducible if it cannot be decomposed into a direct sum $S = S_1 \oplus S_2$ of non-trivial operators.
Under a canonical minimality condition one proves that the spectrum of the normal extension is contained in the spectrum of the original operator:
$$ \sigma(N) \subset \sigma(S),$$
the difference consisting in filling entire bounded ``holes" of the complement, that is
every $U \subset \C \setminus \sigma (N)$  bounded and connected satisfies either $U \cap \sigma(S) = \emptyset$ or $U \subset \sigma(S).$
The authoritative monograph \cite{Conway-1991} accurately exposes the basics of subnormal operator theory. In the sequel we freely use some 
standard terminology (essential spectrum, index, reproducing kernel, trace, ...) well explained there in a unifying context.

A cyclic subnormal operator is necessarily of the following type. Let $\mu$ be a compactly supported, positive Borel measure on the complex plane $\C$. The multiplier
$S_\mu = M_z$ acting on the closure $P^2(\mu)$ of polynomials in Lebesgue space $L^2(\mu)$ is obviously subnormal (the minimal normal extension is $N = M_z$ acting on $L^2(\mu)$)
and has the constant function ${\mathbf 1}$ as a cyclic vector. If the measure $\mu$ is not finitely atomic, then one can speak without ambiguity of the basis
of $P^2(\mu)$ formed by the orthonormal polynomials
$$ P_n(z) = \gamma_n z^n + \ldots \in \C[z], \ \ n \geq 0,$$
where $\gamma_n >0$ and
$$ \langle P_k, P_j \rangle_{2,\mu} = \delta_{jk}, \ \ j,k \geq 0.$$
The spectrum of $N$ is equal to the closed support of $\mu$, while the spectrum of $S_\mu$ contains ${\rm supp} (\mu)$ plus some ``holes" of its complement.

\begin{theorem}[Thomson] Let $\mu$ be a positive Borel measure, compactly supported on $\C$. There exists a Borel partition $\Delta_0, \Delta_1, \ldots $
of the closed support of $\mu$ with the following properties:
\bigskip 

1) $P^2(\mu) = L^2(\mu_0) \oplus P^2(\mu_1) \oplus \ldots$, where $\mu_j = \mu|_{\Delta_j}, \ j \geq 0;$
\bigskip

2) Every operator $S_{\mu_j}, \ j \geq 1,$ is irreducible with spectral picture:
$$ \sigma(S_{\mu_j}) \setminus \sigma_{\rm ess}(S_{\mu_j}) = G_j,\ \ {\it simply \ connected},$$
and
$$ {\rm supp} \mu_j \subset \overline{G_j}, \ \ j \geq 1;$$
\bigskip

3) If $\mu_0 = 0$, then any element $f \in P^2(\mu)$ which vanishes $[\mu]$-a.e. on $G = \cup_j G_j$ is identically zero.

\end{theorem}

The proof appeared in \cite{Thomson-1991}, even for $L^p$ spaces, $1 \leq p < \infty.$  The immediate relevance of Thomson's Theorem for the general theory of 
subnormal operators is already discussed in \cite{Conway-1991} (published almost simultaneously with the original article). We confine to comment briefly the statement.

The operator $S_{\mu_0} = M_z \in {\mathcal L}(L^2(\mu_0))$ is the normal component of $S_\mu$. If $S_\mu$ is not normal, then at least one of the summands $S_{\mu_j}, \ j \geq 1,$
is non-trivial. There can be finitely many, or at most countably many such irreducible summands. The spectral picture described in part 2) of the theorem means that the adjoint
of every operator $S_{\mu_j}$ admits a continuum of eigenvalues of multiplicity one, labelled by the simply connected open set $G_j$:
$$ \lambda \in G_j \ \ \Rightarrow [\ker (S_{\mu_j} -\lambda) = 0, \ \ \ker(S_{\mu_j}^\ast - \overline{\lambda}) = 1],$$
and the range of $S_{\mu_j}-\lambda$ is closed. Moreover, the corresponding eigenvectors span $P^2(\mu_j)$, to the extent that this functional Hilbert space possesses a
reproducing kernel. To be more specific, for every $\lambda \in G_j$ the corresponding point evaluation is bounded as a linear functional on $P^2(\mu_j)$, and hence on
$P^2(\mu)$:
$$ \Lambda^{\mu_j}(\lambda) : = \inf \{ \|p \|^2_{\mu_j}; p \in \C[z], \ p(\lambda) =1\} >0.$$
Above $\Lambda^{\tau}(\lambda)$ is the {\it Christoffel function} associated to the measure $\tau$ and point $\lambda$.
This means that the support of $\mu_j$ must disconnect the simply connected open set $G_j$ from the interior points of its complement.
We will depict some examples in the last section.

The specific aim of the present note is to extract from the subnormal operator $S_\mu$ its non-normal part and relate this operation to the splitting, henceforth called
{\it Thompson's partition}, of the 
original measure $\mu = \mu_0 + [\mu_1 + \mu_2 + \ldots].$ At this stage we turn the page and invoke some distinctive results concerning a wider class of linear operators.
We will give details including terminology in the proof of the following theorem. Below ${\rm dA}$ stands for Lebesgue two dimensional measure and $[\cdot, \cdot]$ denotes the commutator of two operators.

\begin{theorem} Let $\mu = \mu_0 + \mu_1 + \mu_2 + \ldots$ be Thompson's partition of a positive measure $\mu$ with compact support on $\C$. Let $S = S_\mu$
be the corresponding subnormal operator with irreducible parts having spectrum $\sigma(S_{\mu_j}) = \overline{G_j}, \ j \geq 1.$ 

For every pair of non-negative integers $k,\ell$ the following trace exists and equals a moment integral:
\begin{equation}\label{newmoments}
{\rm Tr} [ (S^\ast)^{k+1}, S^{\ell+1} ] = \frac{(k+1)(\ell+1)}{\pi} \sum_j \int_{\overline{G_j}} \overline{z}^k z^{\ell}{\rm dA}(z).
\end{equation}

\end{theorem}

\begin{proof} The subnormal operator $S$ is {\it hyponormal}, that
is its self-commutator is non-negative:
$$ [S^\ast,S] = S^\ast S - S S^\ast  \geq 0.$$
Since $S$ is also cyclic, a theorem due to Berger and Shaw \cite{Berger-Shaw-1974} asserts that the the above self-commutator is also trace class:
\begin{equation}\label{BS}
{\rm Tr} [ S^\ast, S] < \infty.
\end{equation}
Due to non-negativity ${\rm Tr} [ S^\ast, S] =0$ if and only if $S$ is normal. A different proof of Berger-Shaw Theorem is due to Voiculescu \cite{Voiculescu-1980'}, where
the modulus of quasi-triangularity and other perturbation theory concepts enter into discussion.

Hyponormal operators with trace-class self-commutators are modeled by singular integrals (with Cauchy kernel singularity) on Lebesgue space  on the real line. Early on, starting from such functional models Pincus \cite{Pincus-1968}
has identified a spectral invariant called the {\it principal function}. A decade of astounding discoveries led to the following trace formula, which can also be taken as a definition
of the principal function. Let $T \in {\mathcal L}(H)$ be a hyponormal operator with trace-class self-commutator. There exists a function $g_T \in L^1_{\rm comp} (\C,{\rm dA})$ with compact support,
satisfying:
\begin{equation}\label{HH}
{\rm Tr}[p(T,T^\ast), q(T,T^\ast)] = \frac{1}{\pi} \int_{\C} ( \frac{\partial p}{\partial \overline{z}} \frac{\partial q}{\partial {z}} - \frac{\partial q}{\partial \overline{z}} \frac{\partial p}{\partial {z}} ) {\rm dA}(z).
\end{equation}
Above $p, q \in \C[z,\overline{z}]$ are polynomials, and the order of factors $T$, $T^\ast$ in the functional calculus $p(T,T^\ast)$ does not affect the trace, due to relation (\ref{BS}).
Originally, Carey and Pincus have defined $g_T$ via a multiplicative commutator determinant, in its turn inspired by the perturbation determinant in the 1D theory of the phase shift.
The trace formula above appeared in the works of Helton and Howe \cite{Helton-Howe-1975}. Proofs and historical details can be found in the monograph \cite{Martin-Putinar-1989}.
Important for our note is the observation that, for an irreducible hyponormal, but not normal operator $T$, the support of the principal function coincides with its spectrum:
\begin{equation}\label{suppg}
{\rm supp } \  g_T = \sigma(T).
\end{equation}
For a proof see pg. 243 of \cite{Martin-Putinar-1989}. 

To give a sense of principal function, for a point $\lambda$ disjoint of the essential spectrum of $T$ one finds the Fredholm index formula:
\begin{equation}\label{index}
g_T(\lambda) = - {\rm ind} (T-\lambda) = \dim \ker(T^\ast - \overline{\lambda}) - \dim \ker(T - {\lambda}).
\end{equation}
The principal function enjoys a series of functoriality properties which are inherited from its curvature type definition (\ref{HH}). 

Another relevant theorem due to Carey and Pincus \cite{Carey-Pincus-1981} asserts that the principal function of a subnormal operator is integer valued. 

Returning to Thompson's partition setting we can assemble the above observations: consider an irreducible component $S_j = S_{\mu_j}, \ j \geq 1,$ of $S$. According to the index formula (\ref{index}),
$g_{S_j}(\lambda)= 1$ for every $\lambda \in G_j$. Moreover, Thompson's Theorem yields $\sigma(S_j) = \overline{G_j}$, and on the other hand $\sigma(S_j) = {\rm supp} \ g_{S_j}.$
According to the cited Carey and Pincus Theorem we infer
$$ g_{S_j} = \chi_{\overline {G_j}}.$$
Thus, for a fixed integer $j \geq 1$ Helton and Howe formula reads
$${\rm Tr} [ (S_j^\ast)^{k+1}, S_j^{\ell+1} ] = \frac{(k+1)(\ell+1)}{\pi}  \int_{\overline{G_j}} \overline{z}^k z^{\ell}{\rm dA}(z), \ \ k,\ell \geq 0.$$

Since $S$ is a direct sum of these irreducible operators and $S_0$ does not contribute to the trace of commutators, the statement is proved.

\end{proof}

By exploiting another feature of the principal function, namely Berger's cyclic multiplicity inequality
$$ g_T \leq {\rm mult} (T),$$
see \cite{Berger-1981}, we derive the following notable property of the open chambers $G_j$.

\begin{corollary} In the conditions of Thompson's Theorem, the area measure of every
intersection $\overline{G_j} \cap \overline{G_k}, j \neq k,$ is zero.
\end{corollary}

\begin{proof}
Indeed,
$$ g_S = g_{S_1} + g_{S_2} + \ldots \leq 1$$
and
$$ g_S \leq 1, \ \ {\rm d A}-a.e.$$
While $g_{S_j} = \chi_{\overline{G_j}}, \ \ j \geq 1.$
\end{proof}

An array of function theoretic results, some providing answers to long standing open questions, stream from Thompson's partition of a measure and the structure of a cyclic subnormal operator, cf. Chapter VIII of \cite{Conway-1991}. An elaborate  analysis of the limiting  values of elements of $P^2(\mu)$ on the boundary of each chamber $G_j$ is pursued in \cite{ARS-2009}, with applications to the structure of invariant subspaces of the operator blocks $S_j$.

\section{Shape reconstruction} In this section we focus on the special nature of the power moments derived from the commutator formula (\ref{newmoments}) by sketching some known
reconstruction algorithms. The specific nature of the generating measure of the new moments (absolutely continuous, with a bounded weight with respect to Lebesgue 2D measure) simplifies and enhances the reconstruction and approximation procedures.

\subsection{}

First  we resort again to the theory of hyponormal operators, this time subject to a more restrictive self-commutator constraint. Details
covering technical aspects touched below and additional bibliographical comments can be found in the recent lecture notes \cite{Gustafsson-Putinar-2017}.

The starting point is Pincus' bijective correspondence between ``shade functions" $g \in L^1_{\rm comp}(\C, {\rm dA}), \ 0 \leq g \leq 1,$  and irreducible hyponormal operators
$T \in {\mathcal L}(H)$ with rank one self-commutator:
$$ [T^\ast, T] = \xi \langle \cdot, \xi \rangle.$$
The explicit formula linking the two classes is:
\begin{equation}\label{det}
\det [ (T-w) (T^\ast-\overline{z}) (T-w)^{-1}(T^\ast-\overline{z})^{-1}] = $$ $$ 1- \langle (T^\ast - \overline{z})^{-1}\xi, (T^\ast - \overline{w})^{-1}\xi \rangle = $$ $$
\exp ( \frac{-1}{\pi} \int_\C \frac{ g(\zeta) {\rm dA}(\zeta)}{(\zeta-w)(\overline{\zeta}-\overline{z})}).
\end{equation}
Above $z,w$ are originally outside the spectrum of $T$, equal to the support of $g$, but the second identity can be extended to the whole $\C^2$. As a matter of fact this is the multiplicative analog, and equivalent, of Helton and Howe additive formula, identifying $g$ with the principal function $g_T$ of $T$.

In our specific situation of a positive measure $\mu$ carrying its Thompson's partition, we face, returning to the notation of the preceding section, a characteristic function 
$$ g = g_{S_\mu} = \sum_j \chi_{\overline{G_j}}.$$
Its moments appear in relation (\ref{newmoments}):
$$ 
a_{k\ell} = \int_\C \zeta^k \overline{\zeta}^\ell g(\zeta) {\rm dA}(\zeta), \ \ k,l \geq 0,
$$
and will be organized in the exponential of a formal generating series:
$$ E_g(w,z) = \exp [\frac{-1}{\pi} \sum_{k,\ell=0}^\infty \frac{a_{k\ell}}{w^{k+1} \overline{z}^{\ell+1}}].$$
This is of course the power expansion at infinity of the double Cauchy integral appearing in (\ref{det}).
In general we define
$$ E_g(w,z) =  \exp ( \frac{-1}{\pi} \int_\C \frac{ g(\zeta) {\rm dA}(\zeta)}{(\zeta-w)(\overline{\zeta}-\overline{z})}), \ \ z,w \in \C, \ z\neq w.$$

We recall a few of the properties of the {\it exponential transform} $E_g$:
\begin{enumerate}[(a)]

\item The function $E_g$ can be extended by continuity to $\C^2$ by assuming the value $E_g(z,z) = 1$ whenever $ \int_\C \frac{ g(\zeta) {\rm dA}(\zeta)}{|\zeta-z|^2} = \infty;$

\item The function $E_g(w,z)$ is analytic in $w \in \C \setminus {\rm supp} (g)$ and antianalytic in $z \in \C \setminus {\rm supp} (g)$;

\item The kernel $1-E_g(w,z)$ is positive semi-definite in $\C^2$;

\item The behavior at infinity contains as a first term the Cauchy transform of $g$:
$$ E_g(w,z) = \frac{1}{\overline{z}} [\frac{-1}{\pi} \int_\C \frac{ g(\zeta) {\rm dA}(\zeta)}{\zeta-w} + O(\frac{1}{|z|^2}), \ \ |z| \rightarrow \infty.
$$ 
\end{enumerate}

The case of a characteristic function $g= \chi_\Omega$ of a bounded domain $\Omega$ is particularly relevant for our note. In this case we simply write
$E_\Omega$ instead of $E_g$, and we record the following properties (all proved and well commented in \cite{Gustafsson-Putinar-2017}).

\begin{enumerate}[(i)]

\item The equation 
$$ \frac{\partial E_\Omega(w,z)}{\partial \overline{w}} = \frac{E_\Omega(w,z)}{\overline{w}-\overline{z}}$$
holds for $z \in \Omega$ and $w \in \C \setminus \Omega$;

\item The function $E_\Omega(w,z)$ extends analytically/antianalytically from \\
$(\C \setminus \overline{\Omega})^2$ across real analytic arcs of the boundary of  $\Omega$;

\item Assume $\partial \Omega$ is piecewise smooth. Then $z \mapsto E_\Omega(z,z)$ is a superharmonic function on the complement of $\Omega$,
with value $1$ at infinity, vanishing on $\overline{\Omega}$ and satisfying
$$ E_\Omega(z,z) \approx {\rm dist} (z, \partial \Omega)$$
for $z$ close to $\partial \Omega$;
\end{enumerate}

A Riemann-Hilbert factorization also characterizes $E_\Omega$, see \cite{Gustafsson-Putinar-2017}. The feature which turns the exponential transform into a suitable shape reconstruction from moments tool is 
its rationality  on a class of domains which approximate in Hausdorff distance any planar domain. More specifically, a bounded open set $U$ of the complex plane
is called a {\it quadrature domain} for analytic functions if there is a distribution of finite support $\tau$ in $U$ (combination of point masses and their derivatives), such that
$$ \int_U f {\rm d A} = \tau(f), \ \ f \in L^1_a(U, {\rm d A}),$$
where $L^1_a(U, {\rm d A})$ stands for the space of analytic functions in $U$ which are integrable. Such a quadrature domain has always a real algebraic boundary, even irreducible if
$U$ is connected. The simplest example is of course a disk. Any simply connected quadrature domain is a conformal image of the disk by a rational function. The term was coined by Aharonov and Shapiro for specific function theory purposes, but soon it was realized that quadrature domains are relevant in fluid mechanics, geophysics, integrable systems and operator theory. An informative and accessible survey is \cite{Gustafsson-Shapiro-2005}.

Quadrature domains can be exactly reconstructed from moments via formal algebraic manipulations of the exponential transform. Namely, let $d$ be a fixed integer and let
$(a_{k\ell})_{k,\ell=0}^d,$
be a hermitian matrix of potential moments of a ``shade function" $g(z), 0 \leq g \leq 1,$. Consider the truncated exponential transform
$$ F(w,z) =  \exp [\frac{-1}{\pi} \sum_{k,\ell=0}^d \frac{a_{k\ell}}{w^{k+1} \overline{z}^{\ell+1}}] = $$ $$1 + \sum_{m,n=0}^\infty \frac{b_{mn}}{w^{m+1} \overline{z}^{n+1}}.$$
A necessary and sufficient condition that $(a_{k\ell})_{k,\ell=0}^d$ represent the moments of a quadrature domain is the existence of a monic polynomial $P(z)$ of degree $d$ and a rational function of the form
$$ R_d(w,z) = \frac{\sum_{m,n=0}^{d-1} c_{mn} w^m \overline{z}^n}{P(w) \overline{P(z)}},$$
such that, at infinity
$$ F(w,z) - R(w,z) = O( \frac{1}{w^{d+1} \overline{z}^d},  \frac{1}{w^{d} \overline{z}^{d+1}}).$$
The reader will recognize above a typical 2D Pad\'e approximation scheme. Moreover, for any shade function $g$, the exponential transform
$E_g$ coincides with $E_\Omega$, where $\Omega$ is a quadrature domain if and only if 
$$E_g(w,z) = \frac{\sum_{m,n=0}^{d-1} c_{mn} w^m \overline{z}^n}{P(w) \overline{P(z)}}, \ \ |z|, |w| \gg 1.$$
In this case the zeros of $P$ coincide with the quadrature nodes, while the numerator is the irreducible defining polynomial of the boundary of $\Omega$:
$$ \partial \Omega \subset \{ z \in \C; \ \sum_{m,n=0}^{d-1} c_{mn} z^m \overline{z}^n = 0 \}.$$

The above Pad\'e approximation procedure was proposed for the reconstruction of planar shapes in \cite{GHMP-2000} and it is also fully explained in \cite{Gustafsson-Putinar-2017}.

\subsection{}

One can approximate unions of simply connected domains with smooth boundary without going through the costly non-linear exponential transform. This time one can invoke simply the Christoffel-Darboux kernel and its behavior inside, outside or on the boundary of the set. The necessary estimates for this approach appear in \cite{GPSS-2009}, while a refinement of the method, for non-simply connected sets (an ``archipelago of islands with lakes'') appears in \cite{SSST-2015}. The give only one glimpse in this direction, if $\Omega$ is such an archipelago, and $P_n(z)$ denote the associated complex orthogonal polynomials, then Christoffel's function: 
$$ \Lambda^\Omega_n(\lambda) = \inf \{ \| p \|^2_{2,\Omega}, \ p \in \C[z], \ \deg(p) \leq n, \ p(\lambda)=1\},$$ satisfies:
$$\sqrt{\pi} {\rm dist}(z, \partial \Omega) \leq \sqrt{\Lambda^\Omega_n(z)}  \leq C {\rm dist}(z, \partial \Omega)$$
for $z \in \Omega$, close to $\partial \Omega$. Moreover:
$$  \Lambda^\Omega_n(\lambda) = O(\frac{1}{n}), \ \ \lambda \in \Gamma.$$
In the exterior of $\overline{\Omega}$ Christoffel's function decays exponentially to zero.
On analytic boundaries one have sharp estimates, complemented by some computational analysis and graphical experiments, see again \cite{GPSS-2009,SSST-2015}.

\subsection{}

Finally, if it is a priori given that the boundary of an open set $\Omega$ is given by a single real polynomial $q$ of degree $d$, and $\Omega$ is its sublevel set, then Stokes Theorem
(of geometric measure theory) allows to identify $q$ from all power moments of $\Omega$ of degree less than or equal to $3d$. The additional assumption that $\Omega$ is convex
drops this bound to $2d$. For details see \cite{Lasserre-Putinar-2015}. The same idea is generalized there to the reconstruction of algebraic domains carrying an exponential weight.

\section{Exclusion of outliers} We start this concluding section with a few examples, validating, or showing the limits, of the ideal separation of ``outliers'' from the ``cloud'' we propose.
Thompson's decomposition of a positive Boreal measure in 2D strongly uncouples (in orthogonal Hilbert space directions) the normal part of the multiplier by the complex variable from its pure subnormal part. This will be the first step in our scheme. On simple examples, that is a cloud of uniform area mass versus finitely many atoms, the partition  we propose will do as expected. However, on more sophisticated measures, the decomposition might be counterintuitive.

\begin{example} Two measures supported by the unit circle $\T$ produce the same essential spectrum, yet their Thompson's partition is very different. Specifically, let $d\theta$ denote
arc length, and 
$$ \mu = \sum_{k=1}^\infty \frac{1}{2^k} \delta_{z_k},$$
where $(z_k)$ is an everywehere dense sequence on $\T$.

The operator $S_{d\theta}$ is pure subnormal; this is the unilateral shift acting on Hardy space of the disk. The spectrum is the closure of the unit disk,
$$ \sigma(S_{d\theta}) = \overline{\D},$$
the essential spectrum is the boundary
$$ \sigma_{ess} (S_{d\theta}) = \T,$$
and the principal function is the characteristic function of the disk. The associated exponential transform is
$$ E_\D(w,z) = \exp[ \frac{-1}{\pi} \int_\D \frac{{\rm dA}(\zeta)}{(\zeta-w)(\overline{\zeta}-\overline{z})}] = 1 -\frac{1}{w \overline{z}}, \ \ |z|, |w|>1,$$ and indeed
$z=0$ is the quadrature node of $\D$ and $1-z \overline{z} =0$ is the equation defining the boundary.

On the other hand, $S_\mu$ is a normal operator, with spectrum equal to its essential spectrum
$$ \sigma(S_\mu) = \sigma_{ess}(S_\mu) = \T.$$ 
As a matter of fact, $\mu$ can be any singular measure with respect to arc length, and support equal to the full circle.

\end{example}

\begin{example} We remain in the closed unit disk, and consider after Kriete \cite{Kriete-1979} a measure $\mu$ which inside the disk
is rotationally invariant
$$ \mu |_\D = G(r) r dr d \theta $$
while on the boundary has the form
$$ \mu |_\T = w d\theta + \nu_s,$$
where $G,w$ are positive, continuous weights and $\nu_s$ is singular with respect to arc length. In general, $L^2(\nu_s)$ is a direct summand of $P^2(\mu)$, but wether
$L^2(w d\theta)$ is also a direct summand of $P^2(\mu)$ is a challenging question, touching delicate harmonic analysis chapters; a sufficient condition depending on the growth of the function $G$ and the behavior of the
boundary weight $w$. Without entering into details, we mention that there exists an example (attributed by Kriete to Vol'berg) with $G(r) =1, \  r \in [0,1),$ and weight $w(\theta)$
necessarily satisfying
$$ \int_{-\pi}^\pi \log w(\theta) d \theta = -\infty,$$
so that
$$ P^2(\chi_\D {\rm dA} + w d\theta) = L^2_a(\D) \oplus L^2(w d\theta).$$
Above $L^2_a(\D, {\rm dA}) = P^2(\chi_\D {\rm dA})$ is the Bergman space of the unit disk, that is the space of analytic functions in $\Omega$ which are square summable on $\Omega$.

In such a case, our splitting schema will separate the area measure in the disk from the boundary component $w d\theta$. 
\end{example}

Speaking about the Bergman space of a simply connected domain $\Omega$, classical results of Carleman (1923), then later Markushevich and Farrell (1934) prove that
$$ L^2_a(\Omega, {\rm dA}) = P^2(\chi_\Omega {\rm dA}),$$
for a Jordan domain, respectively for a Carath\'eodory domain. In this situation, Thompson's partition has obviously only one chamber, equal to $\Omega$, and no normal part.
By definition, $\Omega$ is a {\it Carath\'eodory domain} if its boundary equals the boundary of the unbounded connected component of 
the complement of its closure. For example, a Carath\'eodory domain does not have internal slits. The early survey \cite{Mergeljan-1962} is invaluable for information on the contributions of the Russian and Armenian schools to this very topics.

Asking the same completeness of polynomials in Bergman space question beyond Carath\'eodory domains turns out to be very challenging, and interesting. It is not our aim to enter into
inherent technical details. We confine to record a recent outstanding observation of Akeroyd \cite{Akeroyd-2011} stating that there exists a slit $\Gamma$ joining an internal point to a boundary point of the unit disk, so that $ L^2_a(\D \setminus \Gamma, {\rm dA}) = P^2(\chi_{\D\setminus \Gamma} {\rm dA}).$ In general, this is not the case, the relative geometry of the curve $\Gamma$ inside $\D$ altering the completeness of polynomials . See \cite{Brennan-1977} for an authoritative text, containing definitive results.

\begin{example}
To get closer to the aim of this note, we have to mention the weighted Bergman space case also. Let $\Omega$ be a simply connected domain and $w: \Omega \longrightarrow (0,\infty)$ a continuous weight. We denote by $L^2_a(\Omega, w {\rm dA})$ the space of analytic functions in $\Omega$ which are square summable with respect to the measure $w {\rm dA}$
on $\Omega$.
 Early theorems of Hedberg \cite{Hedberg-1968} assert that
$$ L^2_a(\Omega, w {\rm dA}) = P^2(\chi_\Omega w {\rm dA}),$$
if either $w(z) = |f(z)|$, where $f$ is an analytic function in $\Omega$ subject to the growth condition
$$ \int_\Omega (|f|^{-\delta} + |f|^{1+\delta}) {\rm d A} < \infty,$$ for some $\delta>0$, or
$$ L^2_a(\D, (w\circ \phi) {\rm dA}) = P^2(\chi_\D (w\circ \phi) {\rm dA}),$$
where $\phi: \D \longrightarrow \Omega$ is a conformal mapping. For instance a positive weight $w$ making the pull-back on the disk $w \circ \phi$ rotationally invariant is appropriate for polynomial density in the weighted Bergman space. In other terms $w$ is constant on the level sets of the inner Green function of $\Omega$. See \cite{Brennan-1979} for complementary results and history of this completeness problem.
\end{example}

After this lengthy and still sketchy preparation for a possible delimitation of the mathematical meaning of a ``cloud'', we are ready to put forward the announced algorithm.
\bigskip

{\bf Algorithm for separation of outliers from a cloud.} {\it Let $\mu = \chi_\Omega w {\rm dA} + \nu$ be a positive measure, where $\Omega$ is a bounded simply connected domain with piece-wise smooth boundary, $w$ is a positive continuous weight on $\Omega$ and $\nu$ is a finite atomic measure. We assume that complex polynomials are dense in the weighted Bergman space
\begin{equation}\label{Bergman-completeness}
 L^2_a(\Omega, w {\rm dA}) = P^2(\chi_\Omega w {\rm dA})
 \end{equation}
 and the power moments 
 $$ s_{k\ell} = \int_\C z^k \overline{z}^\ell d\mu(z), \ \ k,\ell \geq 0,$$
 are known.
 
\begin{enumerate}
\item Compute the associated complex orthogonal polynomials $P_k(z), \ k \geq 0,$ and the corresponding truncated Christoffel-Darboux kernel
$$ K^\mu_d (w,z) = \sum_{j=0}^d P_j(w) \overline{P_j(z)}, \ \ d \geq 0.$$
\bigskip

\item For ever pair of integers $k,\ell \geq 0$ compute the trace of commutator
$$ {\rm Tr}[(S^\ast)^{k+1}, S^{\ell+1}] = $$ $$
\sum_{j=0}^\infty\lim_{d \rightarrow  \infty}  \int_C ( z^{\ell+1}\overline{z}^{k+1} |p_j(z)|^2 -
z^{\ell+1} \overline{p_j(z)}\int_C K^\mu_d(z,\zeta) \overline{\zeta}^{k+1} p_j(\zeta) {d\mu}(\zeta) ){ d\mu}(z)=$$
$$ \sum_{j=0}^\infty\lim_{d \rightarrow  \infty} \int_{\C \times \C} K^\mu_d(z,\zeta) z^{\ell+1}\overline{p_j(z)} [ \overline{z}^{k+1}p_j(z)- \overline{\zeta}^{k+1}p_j(\zeta)]
{d\mu}(\zeta) {d\mu}(z).$$
\bigskip

\item The moments of the ``cloud'' carrying uniform mass equal to one are:
$$ a_{k\ell} = \int_\Omega z^k \overline{z}^\ell {\rm dA}(z) = \frac{\pi}{(k+1)(\ell + 1)} {\rm Tr}[(S^\ast)^{k+1}, S^{\ell+1}], \ \ k,\ell \geq 0.$$
\bigskip

\item Use one of the reconstruction of shapes from moments procedures (outlined in Section 3) to approximate, or find exactly $\Omega$.
\end{enumerate}
}
\bigskip

The limits above, that is the trace of the commutator, exist by Berger and Shaw Theorem. The decomposition
$$ P^2(\mu)  = P^2(\chi_\Omega w {\rm dA}) \oplus L^2(\nu)$$
is trivial, given the finite dimensionality of the space $L^2(\nu)$ and the existence of a polynomial which vanishes exactly at the support of $\nu$.
Notice that in Step 2) solely depends on the moments of the measure $\mu$.
The hypotheses in the algorithm can be obviously be much relaxed, allowing for instance a union of disjoint simply connected domains and 
an array of singular measures $\nu$ (with respect to area) whose support does not 
disconnect the plane.

A simple consequence of the algorithm can be formulated in terms of the associated Hessenberg matrix. We provide one such instance with numerical 
matrix analysis flavor. To be more precise, write, in the conditions and notation adopted in the algorithm:
$$ z p_k(z) = \sum_{n=0}^{k+1} h_{nk} p_n(z), \ \ k \geq 0.$$
The Hessenberg matrix $(h_{nk})_{n,k=0}^\infty$ has only the first sub-diagonal non-zero and it represents the subnormal operator $S=S_\mu$ with respect to the orthonormal basis
$(p_k)_{k=0}^\infty$.

Fix an integer $j \geq 0$. Then
$$ \langle S^\ast S p_j, p_j \rangle = \| S p_j \|^2 = \sum_{n=0}^{j+1} |h_{nj}|^2,$$
while
$$ \langle S S^\ast p_j, p_j \rangle = \| S^\ast p_j \|^2 = \sum_{k=0}^\infty |\langle S^\ast p_j, p_k\rangle |^2 = \sum_{k=0}^\infty |h_{jk}|^2.$$
Hence
$$ {\rm Tr}[S^\ast,S] = \sum_{j=0}^\infty \sum_{k=0}^\infty (|h_{kj}|^2-|h_{jk}|^2).$$
In view of Helton and Howe formula we infer:

\begin{prop}\label{area} Let $(h_{kj})$ denote Hessenberg's matrix associated to the complex orthogonal polynomials in $P^2(\mu)$. If the measure $\mu$ satisfies the conditions in the Algorithm, then the area of the ``cloud'' $\Omega$ can be recovered from the identity:
$$ {\rm Area}(\Omega) = \pi \sum_{j=0}^\infty \sum_{k=0}^\infty (|h_{kj}|^2-|h_{jk}|^2).$$
\end{prop}

Similarly one can derive a formula for the center of mass of $\Omega$:
$$ \int_\Omega z {\rm dA}(z) = \frac{\pi}{2} \sum_{j=0}^\infty \sum_{\ell \leq k+1} h_{\ell k} (h_{kj} \overline{h_{\ell j}} - h_{j\ell} \overline{h_{jk}}).$$

The alert reader will recognize that the ordering of terms and summation in Proposition \ref{area} is crucial. For instance, if 
$$ \sum_{j,k=0}^\infty | h_{jk}|^2 < \infty,$$
that is $S$ is a Hilbert-Schmidt operator, then ${\rm Area}(\Omega)=0$. Recall also that ${\rm Tr}[A,B] = 0$ for any trace-class operator $A$ and bounded operator $B$.
Consequently a trace class perturbation of the Hessenberg matrix $S$ will not affect Step 2) in the Algorithm, and hence the transformed moments $(a_{k\ell})$. Even more, a theorem
of Voiculescu \cite{Voiculescu-OT2} implies the same invariance of traces of commutators (appearing in Step 2)) by any Hilbert-Schmidt perturbation of $S$ subject to $[S^\ast,S]$ being trace-class. We put together these observations in the following statement.

\begin{prop} In the conditions of the Algorithm, a trace-class additive perturbation $\tilde{S}$ of Hessenberg's matrix $S$, or a Hilbert-Schmidt additive perturbation $\tilde{S}$ 
satisfying ${\rm Tr}| [\tilde{S}^\ast, \tilde{S}] | <\infty$, will not alter the output:
$$ {\rm Tr}[(S^\ast)^{k+1}, S^{\ell+1}] = {\rm Tr}[({\tilde S}^\ast)^{k+1}, {\tilde S}^{\ell+1}], \ \ k, \ell \geq 0.$$
\end{prop}

More details about an asymptotic triangularization with respect to the Hilbert-Schmidt class of the operator $S^\ast$ (fulfilling the constraints of the Algorithm) can be found in \cite{Voiculescu-OT2}.

\bibliography{CDbibl}
\bibliographystyle{plain}

\end{document}